\documentclass[12pt,reqno]{amsart}
\usepackage{amsmath,amssymb,amscd,verbatim,color}

\usepackage{amsfonts}
\usepackage[top=35mm, bottom=35mm, left=30mm, right=30mm]{geometry}

\usepackage{xcolor}
\theoremstyle{plain}
\newtheorem{thm}{Theorem}[section]

\newtheorem{lem}[thm]{Lemma}
\newtheorem{prop}[thm]{Proposition}

\theoremstyle{definition}
\newtheorem{defn}[thm]{Definition}

\numberwithin{equation}{section}

\allowdisplaybreaks

\begin{document}
	
	\title[Weak Horseshoe with bounded-gap-hitting times]{Weak Horseshoe with bounded-gap-hitting times}

	%
	%
	\author[L. Xu J. Zheng]{Leiye Xu and Junren Zheng}
	\address{L. Xu: WU WEN-TSUN KEY LABORATORY OF MATHEMATICS, USTC, CHINESE ACADEMY OF SCIENCES AND DEPARTMENT OF MATHEMATIC,UNIVERSITY OF SCIENCE AND TECHNOLOGY OF CHINA, HEFEI, ANHUI, CHINA}

	\email{leoasa@mail.ustc.edu.cn}
	
    \address{J. Zheng: WU WEN-TSUN KEY LABOTATORY OF MATHEMATICS, USTC, CHINESE ACADEMY OF SCIENCES AND DEPARTMENT OF MATHMATIC, UNIVERSITY OF SCIENCE AND TECHNOLOGY OF CHINA, HEFEI, ANHUI, CHINA}

	\email{jrenz@mail.ustc.edu.cn}

	
	\begin{abstract}
     In this paper, we consider weak horseshoe with bounded-gap-hitting times.
    For a flow $(M,\phi)$, it is shown that if the time one map $(M,\phi_1)$ has weak horseshoe with bounded-gap-hitting times, so is $(M,\phi_\tau)$ for all $\tau\neq 0$. In addition, we prove that for an affine homeomorphsim of a compact metric abelian group, positive topological entropy is equivalent to weak horseshoe with bounded-gap-hitting times.
   \end{abstract}

	\maketitle
	
	\section{Introduction}
    Throughout this paper,
    by a {\it topological dynamical system} (TDS for short) we mean a pair $(X,T)$,
    where $X$ is a compact metric space and $T:X\rightarrow X$ is a homeomorphism.
    In the early of 1960s, Smale constructed
    the well-known horseshoe map, which showed that a structurally stable diffeomorphism
    may be very complicated \cite{S65}.
    We say that a TDS $(X,T)$ has a {\it horseshoe}
    if there exists a subsystem $(\Lambda,T^n)$ of $(X,T)$ for some $n\in\mathbb{N}$ which is
    topologically conjugate to the two-sided full shift
    $(\{0,1\}^{\mathbb{Z}},\sigma)$. We say that a TDS $(X,T)$ has a {\it semi-horseshoe}
    if there exists a subsystem $(\Lambda,T^n)$ of $(X,T^n)$ for some $n\in\mathbb{N}$ which is
    topologically semi-conjugate to the two-sided full shift.

    \medskip

   The notion of weak horseshoe first appeared in \cite{HL} (see also \cite{HY,KL07}). We introduce several kinds of weak horseshoes in the following.
   \begin{defn}\label{Weak-horshoe} Let $(X,T)$ be a TDS. Then
   \begin{enumerate}
   \item We say $(X,T)$ has {\it a weak horseshoe with positive-density-hitting times},  if there exist two disjoint closed subsets $U_0,U_1$ of $X$,  a constant $b>0$ and $J\subset  \mathbb{N}$ such that the limit
   \[\lim_{m\rightarrow +\infty}\frac{1}{m}|J\cap \{0,1,2,\cdots, m-1\}|\] exists and is larger than or equal to $b$
   (positive density), and for any $s\in
   \{0,1\}^{J}$, there exists $x_s\in X$ with
   $T^j(x_s)\in U_{s(j)}$ for any
   $j\in J$.

   \item We say $(X,T)$ has {\it a weak horseshoe with bounded-gap-hitting times},  if there exist  two disjoint closed subsets $U_0,U_1$ of $X$,  a constant $K>0$ and a subset $J=\{n_i\}_{i\in \mathbb{Z}}$ of integer numbers with
   $0<n_{i+1}-n_i<K$ for $i\in\mathbb{Z}$,
   and for any $s\in \{0,1\}^{J}$, there exists $x_s\in X$ such that
   $T^{j}(x_s)\in U_{s(j)}$ for any
   $j\in J$.

   \item We say $(X,T)$ has {\it a weak horseshoe with periodic hitting times},  if there exist  two disjoint closed subsets $U_0,U_1$ of $X$,  and $J=\{Kn:n\in \mathbb{Z}\}$ for some positive integer $K\in \mathbb{N}$ such that
   for any $s\in \{0,1\}^{J}$, there exists $x_s\in X$ with
   $T^{j}(x_s)\in U_{s(j)}$ for any
   $j\in J$.

   \end{enumerate}
   \end{defn}

   In this paper, by a {\em flow } we mean a pair $(M,\phi)$, where $M$ is a compact metric space with metric $d$ and $\phi:M\times\mathbb{R}\to X$ is a continuous map on $M\times\mathbb{R}$ satisfying
$\phi_0(x)=x$  and  $\phi_{t}\circ \phi_{s}(x)=\phi_{t+s}(x)$ for all $x\in M$ and $t,s\in \mathbb{R}$,
where $\phi_{t}(x)=\phi(x,t)$ for each $t\in \mathbb{R}$ and $x\in M$. It is known that for a flow $(M,\phi)$, the time one map $(M,\phi_1)$ has a horseshoe does not imply that $(M,\phi_\tau)$  has a horseshoe for all $\tau\neq 0$. Here is an example: Let $\sigma$ be the left shift on $\{0,1\}^\mathbb{Z}$. The {\it suspension flow} $\phi$ , which is constructed over $(\{0,1\}^{\mathbb{Z}},\sigma)$, acts in the space
$$M=\{(x,s):x\in \{0,1\}^{\mathbb{Z}},0\le s< 1\}$$
moves each point in $M$ vertically upward with unit
speed, and we identify the point $(x,1)$ with $(\sigma x,0)$ for any $x\in \{0,1\}^{\mathbb{Z}}$. More precisely, for $(x,s)\in M$ and $t\in\mathbb{R}$
$$\phi_t(x,s)=(\sigma^n x,s+t-n)$$
where $n\in\mathbb{Z}$ is the unique number such that $n\le s+t<n+1$. Then if $\tau$ is an irrational number, $(M,\phi_\tau)$  does not have a horseshoe.

    \medskip

    Given a TDS $(X,T)$, one can define its topological entropy $h_{\text{top}}(T,X)$  in the usual
    way (see \cite{AKM}). For a flow $(M,\phi)$, it is well known  that $h_{\text{top}}(\phi_t)=|t|h_{\text{top}}(\phi_1)$ for any $t\in \mathbb{R}$ (see \cite{Abr}).  In \cite{HY}, Huang and Ye showed that a TDS $(X,T)$ has positive topological entropy if and only if $(X,T)$ has a weak horseshoe with positive density hitting times (see also \cite{HL,KL07}). Thus, it is clear that for a flow $(M,\phi)$, if the time one map $(M,\phi_1)$ has a weak horseshoe with positive density hitting times then $(M,\phi_\tau)$  has the same property for all $\tau\neq 0$. Hence there is a very natural question. Is it true that the time one map $(M,\phi_1)$ has a weak horseshoe with bounded-gap-hitting times implies that $(M,\phi_\tau)$ has the same property for all $\tau\neq 0$?

    \medskip

    The main aim of this paper is to give an affirmative answer to this question. Our main result is as follows:

    \begin{thm}\label{thm-1}
		Let $(M,\phi)$ be a  flow. If $(M,\phi_1)$ has a weak horseshoe with bounded-gap-hitting times, then for all $\tau\neq0$, $(M,\phi_\tau)$ has a weak horseshoe with bounded-gap-hitting times.
	\end{thm}

In the process of studying various kinds of weak horseshoes, we notice that we can define weak horseshoe with bounded-gap-hitting times on a flow, which is similar to the discrete time case. The following is the precise definition.

    \begin{defn}\label{flow bounded gap} Let $(M,\phi)$ be a flow. We say that $(M,\phi)$ has {\it a weak horseshoe with bounded-gap-hitting times},  if there exist two disjoint closed subsets $U_0,U_1$ of $X$,  two constants $K>0$, $L>0$ and a subset $J=\{t_i\}_{i\in \mathbb{Z}}$ of real numbers such that
    $L<t_{i+1}-t_i<K$ for $i\in\mathbb{Z}$,
    and for any $s\in \{0,1\}^{J}$, there exists $x_s\in X$ with
    $\phi_{j}(x_s)\in U_{s(j)}$ for any
    $j\in J$.
    \end{defn}

By the above definition, it is not hard to see that $(M,\phi_1)$ has  a  weak horseshoe with bounded-gap-hitting times implies $(M,\phi)$ has a  weak horseshoe with bounded-gap-hitting times. By using the similar method in the proof of Theorem \ref{thm-1}, we can prove that the inverse direction also holds.

    \begin{thm} \label{thm-2} Let $(M,\phi)$ be a flow. $(M,\phi)$ has  a weak horseshoe with bounded-gap-hitting times if and only if $(M,\phi_1)$ has a weak horseshoe with bounded-gap-hitting times.
    \end{thm}

    In \cite{HLXY}, Huang, Li, Xu and Ye showed that a TDS $(X,T)$ has a semi-horseshoe if and only if it has a weak horseshoe with periodic hitting times. Furthermore, they showed that if an automorphism $T$ on a compact metric abelian group $X$ has positive topological entropy, then it has a semi-horseshoe. In this paper, we study the dynamics of the affine homeomorphism on a compact metric abelian groups and obtain the following result.

    \begin{thm} \label{thm-3} Let $T$ be an affine homeomorphsim of a compact metric abelian group $X$.
    $(X,T)$ has positive topological entropy $($i.e., $h_{\text{top}}(T,X)>0$$)$ if and only if
    $(X,T)$ has a weak horseshoe with bounded-gap-hitting times.
    \end{thm}
\medskip
We also want to know whether or not $(X,T)$ is uniquely ergodic when  $(X,T)$ has a weak horseshoe with bounded-gap-hitting times. We give an answer to this question.

\begin{thm}  \label{thm-4} Let $(X,T)$ be a TDS. If $(X,T)$ has a weak horseshoe with bounded-gap-hitting times, then it has infinitely many minimal sets. Particularly, it is not uniquely ergodic.
\end{thm}
    \medskip

    The paper is organized as follows. In Section 2, we will introduce a useful Lemma and prove Theorem \ref{thm-1} as well as Theorem \ref{thm-2}. In Section 3, we will prove Theorem \ref{thm-3}. In Section 4, we prove Theorem \ref{thm-4}. In Section 5, we will pose some open questions.

    \medskip

    Acknowledgement: Leiye Xu is partially supported by NNSF of China (11801538, 11871188). Junren Zheng is partially supported by NNSF of China (11971455). The authors would like to thank Wen Huang and Xiao Ma for  making many valuable suggestions.

\section{Proof of Theorem \ref{thm-1} and \ref{thm-2} }

In this section, we are to prove Theorem \ref{thm-1} and Theorem \ref{thm-2}. To prove Theorem 1.2, we need to introduce several concepts. A subset $F$ of $\mathbb{Z}$ is said to be {\it thick}   if for any $n\in \mathbb{N}$  there exists $m\in \mathbb{Z}$ such that $\{m, m+1,\cdots, m+n\}\subset F$. An infinite subset $F=\{s_i\}_{i\in \mathbb{Z}}$ of $\mathbb{Z}$ is said to be {\it syndetic} if there exists $K>0$ such that $0<s_{i+1}-s_n\le K$ for all $i\in \mathbb{Z)}$. A subset of $\mathbb{Z}$ is called {\it piecewise syndetic} if it is intersection of a thick subset of $\mathbb{Z}$ and a syndetic subset of $\mathbb{Z}$. The following Lemma is useful for our proof (see for example \cite[Theorem 1.23]{F}).
\begin{lem} \label{lem-1}
Let $\mathbb{Z}=B_1\cup B_2\cup \cdots \cup B_q$ be a partition of $\mathbb{Z}$ into finitely many sets. Then at least one of these sets is piecewise syndetic.
\end{lem}

\begin{defn} Let $(X,T)$ be a TDS and $U_0,U_1$ be two disjoint non-empty closed subsets of $X$. We define
\begin{align*}
Ind_T(U_0,U_1)=&\{ J\subset \mathbb{Z}:  \text{   for any } s\in \{0,1\}^{J},  \text{ there exists } x_s\in X \\
 &\,\text{ with }T^{j}(x_s)\in U_{s(j)} \text{ for any } j\in J\} \text { and }\\
 \Sigma_T(U_0,U_1)=&\{ 1_J\in \{0,1\}^{\mathbb{Z}}: J\in Ind_T(U_0,U_1)\},
 \end{align*}
 where $1_J(i)=1$ if $i\in J$, and $1_J(i)=0$ if $i\not\in J$.
\end{defn}
It is clear that $\Sigma_T(U_0,U_1)$ is a non-empty $\sigma$-invariant closed subset of $\{0,1\}^{\mathbb{Z}}$. Hence, for a subset $F$ of $\mathbb{Z}$,
if for each $n\in \mathbb{N}$ we can find  $r_n\in \mathbb{Z}$ such that $$r_n+F\cap [-n,n]\in Ind_T(U_0,U_1),$$ then $F\in Ind_T(U_0,U_1)$.

Now we are ready to prove Theorem  \ref{thm-1}.
	\begin{proof}[Proof of Theorem  \ref{thm-1}] We will obtain a piecewise syndetic set by using Lemma \ref{lem-1}.
One useful obeservation is that we can get a syndetic set from a piecewise syndetic set.

Let $(M,\phi)$ be a  flow such that $(M,\phi_1)$ has a weak horseshoe with bounded-gap-hitting times. Since $\phi_1$ has a weak horseshoe with bounded-gap-hitting times, there exist two disjoint non-empty closed subsets $U_0,U_1$ of $M$, a syndetic set $J=\{a_i\}_{i\in \mathbb{Z}}$ of $\mathbb{Z}$ such that $J\in Ind_{\phi_1}(U_0,U_1)$, that is, for any $s\in \{0,1\}^{J}$, there exists $x_s\in M$ with $\phi_1^j(x_s)\in U_{s(j)}$ for any $j\in J$.
		Since $U_0,U_1$ are two disjoint closed subsets of $M$, there exists $\delta>0$ such that $W_0\cap W_1=\emptyset$, where $W_0=\overline{B(U_0,\delta)}$ and $W_1=\overline{B(U_1,\delta)}$.

Now we fix a real number $\tau\neq0$. It is clear that there exists $0<\tau_0<|\tau|$ such that $d(\phi_t(x),x)<\delta$ for all $x\in M$ and $t\in[-\tau_0,\tau_0]$.
		For $k\in\mathbb{N}$, we let $$ A_k=\bigcup_{i\in\mathbb{Z}}\big[a_i+k\tau_0,a_i+(k+1)\tau_0\big].$$
	Since $J$ is syndetic, we can find a $K\in \mathbb{N}$ such that $0<a_{i+1}-a_{i}\le K$ for all $i\in \mathbb{Z}$.
Thus there exists $N\in\mathbb{N}$ such that $\mathbb{R}=\bigcup_{k=1}^{N}A_k$.
		Let $B_k=\{n\in\mathbb{Z}: n\tau\in A_k\}$. Then $\mathbb{Z}=\bigcup_{k=1}^{N}B_k$.
		By Lemma \ref{lem-1}, there exists some $B_m,1\leq m \leq N$ such that $B_m$ is piecewise syndetic.
		Thus there exists a syndetic set $F=\{b_i\}_{i\in \mathbb{Z}}\subset\mathbb{Z}$
		such that for any $n\in \mathbb{N}$ there exists $p_n\in\mathbb{Z}$
		satisfying
		$$p_n+ \{b_{-n},\dotsc,b_{-1},b_0,b_1,\cdots,b_n\}\subset B_m.$$
		We can also require that either $\{p_n\}$ is strictly increasing and $p_n+b_n<p_{n+1}+b_{-(n+1)}$ for all $n\in \mathbb{N}$, or
         $\{p_n\}$ is strictly decreasing and $p_n+b_{-n}>p_{n+1}+b_{(n+1)}$ for all $n\in \mathbb{N}$.

        Let $V_{i}=\phi_{m\tau_0}(W_{i})$ for $i=0,1$. Since $W_0\cap W_1=\emptyset$, we have that  $V_0, V_1$ are two two disjoint non-empty closed subsets  of $M$. Now we are going to show that $F\in Ind_{\phi_\tau}(V_0,V_1)$, that is,  for any $s\in \{0,1\}^{F}$, there exists $x_s\in M$ with
$\phi_\tau^{b_i}x_s\in V_{s(b_i)}$ for all $i\in \mathbb{Z}$. This implies that $(M,\phi_\tau)$ has a weak horseshoe with bounded-gap-hitting times.

		Given $i\in \mathbb{Z}$. Since $b_i+p_n\in B_m,n\geq |i|$ by the definition of $B_m$,  we can find $j(i,n)\in \mathbb{Z}$ such that $$a_{j(i,n)}+m\tau_0\leq(b_i+p_n)\tau\leq a_{j(i,n)}+(m+1)\tau_0$$ for each $n\ge |i|$.		Since $0<\tau_0<|\tau|$, it follows that $\{a_{j(i,n)}\}_{n\ge |i|}\cap\{a_{j(i',n')}\}_{n'\ge |i'|}=\emptyset$ for any $i\ne i'\in \mathbb{Z}$.

For any  $s\in \{0,1\}^F$, we take $s'\in \{0,1\}^J$ such that $s'(a_{j(i,n)})=s(b_i)$ for any $i\in \mathbb{Z}$ and $n\ge |i|$.
Since $J\in Ind_{\phi_1}(U_0,U_1)$, there exists $x_{s'}\in M$ with
$$\phi_{a_{j(i,n)}}(x_{s'})=\phi_1^{a_{j(i,n)}}(x_{s'})\in U_{s'(a_{j(i,n)})}=U_{s(b_i)}$$
for any $i\in \mathbb{Z}$ and $n\ge |i|$.

Next we take a subsequence $\{ n_1<n_2<\cdots\}$ of $\mathbb{N}$ such that  $\lim_{k\to\infty} \phi^{p_{n_k}}_\tau x_{s'}=x_s$ for some $x_s\in M$.
		Then for $i\in\mathbb{Z}$, we have
		\begin{align*}
		d\bigl(\phi_{-m\tau_0}(\phi_\tau^{b_i} x_s), U_{s(b_i)}\bigr)
		&=\lim_{k\to\infty} d\bigl(\phi_{-m\tau_0}(\phi_\tau^{b_i+p_{n_k}} x_{s'}), U_{s(b_i)}\bigr)\\
		&=\lim_{k\to\infty} d\bigl(\phi_{\tau(b_i+p_{n_k})-m\tau_0} x_{s'}, U_{s(b_i)}\bigr)\\
		&\leq
		\limsup_{k\to\infty} d\bigl(\phi_{\tau(b_i+p_{n_k})-m\tau_0} x_{s'}, \phi_{a_{j(i,n_k)}} x_{s'}\bigr)\\
		&\qquad + \limsup_{k\to\infty} d\bigl(\phi_{a_{j(i,n_k)}} x_{s'},
		 U_{s(b_i)}\bigr)\\
		&\leq
		\limsup_{k\to\infty} \max_{0\le t\le \tau_0} d\bigl(\phi_{t+a_{j(i,n_k)}} x_{s'}, \phi_{a_{j(i,n_k)}} x_{s'}\bigr)\\
		&\qquad + \limsup_{k\to\infty} d\bigl(\phi_{a_{j(i,n_k)}} x_{s'},
		U_{s(b_i)}\bigr) \\
		&\leq
		\limsup_{k\to\infty} \max_{0\le t\le \tau_0} d\bigl(\phi_t(\phi_{a_{j(i,n_k)}} x_{s'}), \phi_{a_{j(i,n_k)}} x_{s'}\bigr)\\
		&\qquad + \limsup_{k\to\infty} d\bigl(\phi_{a_{j(i,n_k)}} x_{s'},
		 U_{s(b_i)}\bigr)\\
		&\leq \delta +0=\delta.
		\end{align*}
		Hence $\phi_{-m\tau_0} (\phi_\tau^{b_i} x_s) \in  W_{s(b_i)}$. It follows that $\phi_\tau^{b_i} x_s\in V_{s(b_i)}$. Thus we show that $F\in Ind_{\phi_\tau}(V_0,V_1)$. This finishes the proof of Theorem  \ref{thm-1}.
	\end{proof}
	
Now we are ready to prove Theorem  \ref{thm-2}.
	\begin{proof}[Proof of Theorem  \ref{thm-2}] We follow the similar arguments of the proof of Theorem \ref{thm-1}.
    Let $(M,\phi)$ be a flow having a weak horseshoe with bounded-gap-hitting times.
        Thus there exist two disjoint closed subsets $U_0,U_1$ of $X$,  two constants $K>0$,$L>0$ and a subset $J=\{t_i\}_{i\in \mathbb{R}}$ of real number such that
    $L<t_{i+1}-t_i<K$ for $i\in\mathbb{R}$,
    and for any $s\in \{0,1\}^{J}$, there exists $x_s\in X$ with
    $\phi_{j}(x_s)\in U_{s(j)}$ for any
    $j\in J$. Since $U_0,U_1$ are two disjoint closed subsets, there exists $\delta>0$ such that $W_0\cap W_1=\emptyset$, where $W_0=\overline{B(U_0,\delta)}$ and $W_1=\overline{B(U_1,\delta)}$.

It is clear that there exists $0<\tau_0<1$ such that $d(\phi_t(x),x)<\delta$ for any $x\in M$ and $|t|<\tau_0$. For $k\in\mathbb{N}$, we let
		 $$ A_k=\bigcup_{i\in\mathbb{Z}}\big[t_i+k\tau_0,t_i+(k+1)\tau_0\big].$$
	Since  $L<t_{i+1}-t_{i}\le K$ for all $i\in \mathbb{Z}$, there exists $N\in\mathbb{N}$ such that $\mathbb{R}=\bigcup_{k=1}^{N}A_k$.
		Let $B_k=\mathbb{Z}\cap A_k$. Then $\mathbb{Z}=\bigcup_{k=1}^{N}B_k$.
		By Lemma \ref{lem-1}, there exists some $B_m,1\leq m \leq N$ such that $B_m$ is piecewise syndetic.
		Thus there exists a syndetic set $F=\{b_i\}_{i\in \mathbb{Z}}\subset\mathbb{Z}$
		such that for any $n\in \mathbb{N}$, there exists $p_n\in\mathbb{Z}$
		satisfying
		$$p_n+ \{b_{-n},\dotsc,b_{-1},b_0,b_1,\cdots,b_n\}\subset B_m.$$
		We can also require that either $\{p_n\}$ is strictly increasing and $p_n+b_n<p_{n+1}+b_{-(n+1)}$ for all $n\in \mathbb{N}$, or
         $\{p_n\}$ is strictly decreasing and $p_n+b_{-n}>p_{n+1}+b_{(n+1)}$ for all $n\in \mathbb{N}$.

        Let $V_{i}=\phi_{m\tau_0}(W_{i})$ for $i=0,1$. Since $W_0\cap W_1=\emptyset$, we have that  $V_0, V_1$ are two two disjoint non-empty closed subsets  of $M$. Now we are going to show that $F\in Ind_{\phi_1}(V_0,V_1)$, that is,  for any $s\in \{0,1\}^{F}$, there exists $x_s\in M$ with
$\phi_1^{b_i} x_s\in V_{s(b_i)}$ for all $i\in \mathbb{Z}$. This implies that $(M,\phi_1)$ has a weak horseshoe with bounded-gap-hitting times.

		Given $i\in \mathbb{Z}$. Since $b_i+p_n\in B_m,n\geq |i|$ by the definition of $B_m$,  we can find $j(i,n)\in \mathbb{Z}$ such that $$t_{j(i,n)}+m\tau_0\leq(b_i+p_n)\leq t_{j(i,n)}+(m+1)\tau_0$$ for each $n\ge |i|$.		Since $0<\tau_0<1$, it follows that $\{t_{j(i,n)}\}_{n\ge |i|}\cap\{t_{j(i',n')}\}_{n'\ge |i'|}=\emptyset$ for any $i\ne i'\in \mathbb{Z}$.

For any  $s\in \{0,1\}^F$, we take $s'\in \{0,1\}^J$ such that $s'(t_{j(i,n)})=s(b_i)$ for any $i\in \mathbb{Z}$ and $n\ge |i|$.
By the assumption there exists $x_{s'}\in M$ with
$$\phi_{t_{j(i,n)}}(x_{s'})\in U_{s'(t_{j(i,n)})}=U_{s(b_i)}$$
for any $i\in \mathbb{Z}$ and $n\ge |i|$.

Next we take a subsequence $\{ n_1<n_2<\cdots\}$ of $\mathbb{N}$ such that  $\lim_{k\to\infty} \phi^{p_{n_k}}_1 x_{s'}=x_s$ for some $x_s\in M$.
		Then for $i\in\mathbb{Z}$, we have
		\begin{align*}
		d\bigl(\phi_{-m\tau_0}(\phi_1^{b_i} x_s), U_{s(b_i)}\bigr)
		&=\lim_{k\to\infty} d\bigl(\phi_{-m\tau_0}(\phi_1^{b_i+p_{n_k}} x_{s'}), U_{s(b_i)}\bigr)\\
		&=\lim_{k\to\infty} d\bigl(\phi_{b_i+p_{n_k}-m\tau_0} x_{s'}, U_{s(b_i)}\bigr)\\
		&\leq
		\limsup_{k\to\infty} d\bigl(\phi_{b_i+p_{n_k}-m\tau_0} x_{s'}, \phi_{t_{j(i,n_k)}} x_{s'}\bigr)\\
		&\qquad + \limsup_{k\to\infty} d\bigl(\phi_{t_{j(i,n_k)}} x_{s'},
		 U_{s(b_i)}\bigr)\\
		&\leq
		\limsup_{k\to\infty} \max_{0\le t\le \tau_0} d\bigl(\phi_{t+t_{j(i,n_k)}} x_{s'}, \phi_{t_{j(i,n_k)}} x_{s'}\bigr)\\
		&\qquad + \limsup_{k\to\infty} d\bigl(\phi_{t_{j(i,n_k)}} x_{s'},
		U_{s(b_i)}\bigr) \\
		&\leq
		\limsup_{k\to\infty} \max_{0\le t\le \tau_0} d\bigl(\phi_t(\phi_{t_{j(i,n_k)}} x_{s'}), \phi_{t_{j(i,n_k)}} x_{s'}\bigr)\\
		&\qquad + \limsup_{k\to\infty} d\bigl(\phi_{t_{j(i,n_k)}} x_{s'},
		 U_{s(b_i)}\bigr)\\
		&\leq \delta +0=\delta.
		\end{align*}
		Hence $\phi_{-m\tau_0} (\phi_1^{b_i} x_s) \in  W_{s(b_i)}$. It follows that $\phi_1^{b_i} x_s\in V_{s(b_i)}$. Thus we show that $F\in Ind_{\phi_1}(V_0,V_1)$. This finishes the proof of Theorem  \ref{thm-2}.
\end{proof}
	
	\section{ Proof of Theorem \ref{thm-3}}
	
    In this section we are to prove Theorem \ref{thm-3}. 	It is showed in \cite{HLXY} that a TDS $(X,T)$ has a semi-horseshoe if and only if it has a weak horseshoe with periodic hitting times. The following result also comes from \cite{HLXY}.
	
	\begin{prop}\label{prop:automorphism}
		If an automorphism $T$ on a compact metric abelian group $X$  has positive topological entropy,
then it has a semi-horseshoe.
	\end{prop}
It is clear that if a TDS $(X,T)$ has a  weak horseshoe with bounded-gap-hitting times, then it has positive entropy. Now Theorem  \ref{thm-3} comes from the following Theorem.
	\begin{thm}\label{thm:algebra-system}
		If an affine homeomorphism $T$ on a compact metric abelian group $X$  has positive topological entropy,
		then it has a  weak horseshoe with bounded-gap-hitting times.
	\end{thm}
	\begin{proof} Let $T(x)=gA(x)$ for all $x\in X$, where $A$ is an automorphism of $X$ and $g\in X$. It is known that $h_{\text{top}}(T)=h_{\text{top}}(A)>0$ (see \cite{B}). Thus by proposition \ref{prop:automorphism}, $A$ has a semi-horseshoe. Hence  $A$ has a weak horseshoe with periodic hitting times, that is,
there exist two disjoint closed subsets $V_0,V_1$ of $X$, and $J=\{Kn:n\in \mathbb{Z}\}$ for some positive integer $K\in \mathbb{N}$ such that for any $s\in \{0,1\}^{J}$, there exists $x_s\in X$ with $A^{j}(x_s)\in V_{s(j)}$ for any $j\in J$, i.e., $J\in Ind_A(V_0,V_1)$.

Let $W$ be an open neighborhood of $e$ where $e$ is the identity element of $X$ such that
$\overline{WV_0}\cap \overline{WV_1}= \emptyset$. Note that
$\mathcal{U}=\{hW: h\in X\}$ is an open cover of $X$. Since $X$ is compact, there exist $h_1,\cdots,h_k\in X$ such that $\bigcup_{i=1}^k h_iW=X$.
For $n\in \mathbb{Z}$, we let $g_n=T^n(e)$ where $e$ is the identity of $X$. It is not hard to see that
$T^n(x)=g_nA^{n}(x)$ for any $x\in X$. Let $B_i=\{n\in \mathbb{Z}: g_{Kn}\in h_iW\}$. Then
$\mathbb{Z}=\cup_{i=1}^{k}B_i$ since $\bigcup_{i=1}^k h_iW=X$.
By lemma \ref{lem-1}, there exists  $1\le m \le k$ such that $B_m$ is piecewise syndetic.
Therefore, there exists a syndetic set $F=\{b_i\}_{i\in \mathbb{Z}} \subset\mathbb{Z}$
such that for any $n\in \mathbb{N}$ there exists $p_n\in\mathbb{Z}$
satisfying
$$p_n+ \{b_{-n},\cdots, b_{-1},b_0,b_1,\dotsc,b_n\}\subset B_m.$$
Let $U_i=\overline{h_mWV_i}$ for $i=0,1$. $U_0$, $U_1$ are two disjoint closed subsets of $X$ since $\overline{WV_0}\cap \overline{WV_1}= \emptyset$. In the following we are to show that $KF\in Ind_T(U_0,U_1)$. This implies that $(X,T)$ has a  weak horseshoe with bounded-gap-hitting times, since $KF$ is a syndetic subset of $\mathbb{Z}$.

To show that $KF\in Ind_T(U_0,U_1)$, it is sufficient to show that for each $n\in \mathbb{N}$, $J_n\in Ind_T(U_0,U_1)$, where $J_n=K(p_n+ \{b_{-n},\cdots, b_{-1},b_0,b_1,\dotsc,b_n\})$. Given $n\in \mathbb{N}$, then $g_{j}\in h_mW$ for all $j\in J_n$. For any  $s\in \{0,1\}^{J_n}$, we take $s'\in \{0,1\}^J$ such that $s'(j)=s(j)$ for any $j\in J_n$.  Since $J\in Ind_A(V_0,V_1)$, there exists $x_{s'}\in X$ with
$A^{j}(x_{s'})\in V_{s'(j)}=V_{s(j)}$ and so $T^{j}(x_s')=g_jA^{j}(x_{s'})\in h_mWV_{s(j)}\subset U_{s(j)}$  for all $j\in J_n$.
Thus $J_n \in Ind_T(U_0,U_1)$. This finishes the proof of Theorem \ref{thm:algebra-system}.
	\end{proof}

\section{Proof of Theorem  \ref{thm-4}}
In this section we are to prove Theorem \ref{thm-4}. Firstly we need the following result, which is essentially Theorem 5.12 in \cite{HLY}.
\begin{prop}  \label{prop-2} Let $p,\ell\in \mathbb{N}$ with $p\ge  2$. For any integer $m \ge 4\ell + 2$, given any
sequence $\{A_n\}_{n\in \mathbb{Z}}$ of subsets of $\{0,1,\cdots,p-1\}^m$ with $|A_n| \le \ell$  for each $n\in \mathbb{Z}$, there exists
$x\in \{0,1,\cdots,p-1\}^{\mathbb{Z}} $ such that $x[n, n + m-1] \not \in  A_n$ for every $n\in \mathbb{Z}$.
\end{prop}

We are ready to proof Theorem  \ref{thm-4}.
\begin{proof}[Proof of Theorem  \ref{thm-4}] Let $(X,T)$ have a weak horseshoe with bounded-gap-hitting times. Then there exist  two disjoint closed subsets $V_0,V_1$ of $X$,  a constant $K>0$ and a subset $J=\{n_i\}_{i\in \mathbb{Z}}$ of integer number such that
   $0<n_{i+1}-n_i<K$ for $i\in\mathbb{Z}$,  and $J\in Ind_T(V_0,V_1)$, that is, for any $s\in \{0,1\}^{J}$, there exists $x_s\in X$ with
   $T^{j}(x_s)\in V_{s(j)}$ for any $j\in J$.

Next we take two disjoint open subsets $U_0$ and $U_1$ of $X$ such that $V_0\subset U_0$ and $V_1\subset U_1$. Then let  $\pi: X\to \{0,1\}$ such that $\pi(x)=0$ if $x\in U_0$ and $\pi(x)=1$ if $x\in U_0^c$. Next we argue by contradiction. Assume that $(X,T)$ has only finitely many minimal sets (named by $X_1,X_2,\cdots, X_r$).

Let $m\ge 4\ell+2$ be as in Proposition \ref{prop-2} for $p=2$ and $l=rK$. For $i=1,2,\cdots,r$, we take $x_i\in X_i$ and $a_i\in \{0,1\}^{mK}$ such that $a_i(j)=\pi(T^jx_i), j=0,1,\cdots mK-1$. For $j\in \mathbb{Z}$, we set $A_j$ to be the subset of $\{0,1\}^{m}$ containing the elements of the form $$(a_i(k),a_i(k+n_{j+1}-n_j),\cdots,a_i(k+n_{j+m-1}-n_j))$$ for $1\le k\le K$ and $1\le i\le r$. Then $|A_j|\le rK=l$ for all $j\in\mathbb{Z}$.

By Proposition \ref{prop-2}, we can take $a\in \{0,1\}^{\mathbb{Z}}$ such that $a[j,j+m-1]\notin A_j$ for all $j\in\mathbb{Z}$. Since $J\in Ind_T(V_0,V_1)$, we can find $y\in X$ such that $T^{n_i}y\in V_{a(i)}$ for all $i\in\mathbb{Z}$. Notice that the closure of the orbit  of $y$ containing at least one minimal set $X_{r_*}$ for some $1\le r_*\le r$. Take $\epsilon>0$ such that $$\epsilon<\min_{z_1\in V_0,z_2\in U_0^c}d(z_1,z_2)\text{ and }\epsilon<\min_{z_1\in V_1,z_2\in U_1^c}d(z_1,z_2).$$ Take $\delta>0$ small enough such that when $z_1,z_2\in X$ with $d(z_1,z_2)<\delta$, one has $d(T^nz_1,T^nz_2)<\epsilon$ for $0\le n\le mK$. Since $X_{r_*}$ is containing in the closure of the orbit  of $y$, there exists some $q\in \mathbb{Z}$ such that $d(T^qy,x_{r_*})<\delta$. Take $j_*\in \mathbb{Z}$ such that  $n_{j_*-1}< q\le n_{j_*}$. Set $k_*=n_{j_*}-q.$ For $0\le j'\le m-1$, since $d(T^qy,x_{r_*})<\delta$ and $0\le k_*+n_{j_*+j'}-n_{j_*}\le mK$, one has $d(T^{k_*+n_{j_*+j'}-n_{j_*}}x_{r_*},T^{k_*+n_{j_*+j'}-n_{j_*}+q}y)<\epsilon$ which implies $T^{k_*+n_{j_*+j'}-n_{j_*}}x_{r_*}\in U_{a(j_*+j')}$ by the fact $T^{k_*+n_{j_*+j'}-n_{j_*}+q}y=T^{n_{j_*+j'}}y\in V_{a(j_*+j')}$. Then $$a_{r_*}(k_*+n_{j_*+j'}-n_{j_*})=\pi(T^{k_*+n_{j_*+j'}-n_{j_*}}x_{r_*})=a(j_*+j')$$ for $0\le j'\le m-1$. That is \begin{align*}&\hskip0.5cm (a(j_*),a(j_*+1),\cdots,a(j_*+m-1))\\
&=(a_{r_*}(k_*),a_{r_*}(k_*+n_{j_*+1}-n_{j_*}),\cdots,a_{r_*}(k_*+n_{j_*+m-1}-n_{j_*}))\in A_{j_*},
 \end{align*}
 which is a contradiction. This finishes the proof of Theorem  \ref{thm-4}.
\end{proof}

	\section{ Some related open problems}
In this section, we are going to mention some related open problems. First of all, by the definition, it is easy to see that  $(X,T)$ has a weak horseshoe with periodic hitting times implies that $(X,T)$ has a weak horseshoe with bounded-gap-hitting times, and $(X,T)$ has a weak horseshoe with bounded-gap-hitting times implies that $(X,T)$ has a weak horseshoe with positive-density-hitting times. Thus it is a natural question whether or not  $(X,T)$ has  a weak horseshoe with bounded-gap-hitting times implies that $(X,T)$ a weak horseshoe with periodic hitting times.

Anther open problem is that whether or not weak horseshoe with bounded-gap-hitting times is an invariant of equivalent flows. Two flows defined on a smooth manifold are equivalent if there exists a homeomorphism of the manifold that sends each orbit of one flow onto an orbit of the other flow while preserving the time orientation. In \cite{ohno}, Ohno constructed a counterexample of equivalent flows with fixed points to indicate that neither zero nor infinity topological entropy is preserved by equivalence.  Furthermore, in \cite{sun}, the authors proved that zero topological entropy is not an invariant of equivalent differentiable flow. However, we do not know whether weak horseshoe with bounded-gap-hitting times is an invariant of equivalent flows.

\end{document}